\documentclass[English]{amsart}%

\usepackage[english]{babel}
\usepackage{amsmath,amssymb,verbatim,amsthm,mathrsfs,mathtools}
\usepackage[bookmarks,hidelinks]{hyperref}

\usepackage[capitalize]{cleveref}
\usepackage{nicefrac}
\usepackage{fullpage}

\renewcommand{\hbar}{\bar{h}}

\DeclareMathOperator{\res}{{\upharpoonright}}

\newtheorem{thm}{Theorem}[section]
\newtheorem{prop}[thm]{Proposition}
\newtheorem{lem}[thm]{Lemma}

\theoremstyle{definition}
\newtheorem{defn}[thm]{Definition}

\theoremstyle{remark}

\DeclareMathOperator{\tp}{tp}

\newcommand{\Lc}{\mathcal{L}}
\newcommand{\Sc}{\mathcal{S}}

\newcommand{\Sbf}{\mathbf{S}}

\newcommand{\ol}{\overline}

\def\Ind{\setbox0=\hbox{$x$}\kern\wd0\hbox to 0pt{\hss$\mid$\hss}
\lower.9\ht0\hbox to 0pt{\hss$\smile$\hss}\kern\wd0}

\def\Notind{\setbox0=\hbox{$x$}\kern\wd0\hbox to 0pt{\mathchardef
\nn=12854\hss$\nn$\kern1.4\wd0\hss}\hbox to
0pt{\hss$\mid$\hss}\lower.9\ht0 \hbox to 0pt{\hss$\smile$\hss}\kern\wd0}

\def\ind{\mathop{\mathpalette\Ind{}}}

\newcommand{\indwo}{\ind^{\!\!\textnormal{wo}}}

\begin{document}

\title{Uniqueness of constructible models in continuous logic}
\author{James E. Hanson}
\address{Department of Mathematics \\
  Iowa State University \\
  396 Carver Hall \\
  411 Morrill Road \\
  Ames, IA 50011, USA}
\email{jameseh@iastate.edu}
\date{\today}


\begin{abstract}
  We show that constructible models of arbitrary complete continuous first-order theories are unique up to isomorphism.
\end{abstract}

\maketitle

\section*{Introduction}

A classic result of early model theory is Vaught's theorem that for a countable complete theory $T$, the following are equivalent:
\begin{itemize}
\item $T$ has a prime model.
\item $T$ has an atomic model.
\item For every $n$, atomic types are dense in $S_n(T)$.
\end{itemize}
Moreover, a countable model is prime if and only if it is atomic, any two prime models are isomorphic, and every prime model is $\aleph_0$-homogeneous \cite{Vaught1959}.

In the context of uncountable theories, the story becomes far more complicated. Primeness and atomicity are no longer equivalent, and Shelah produced several examples of theories that have non-unique prime models \cite{Shelah1979}. Earlier positive results regarding uncountable theories were specifically in the context of $\aleph_0$\nobreakdash-\hspace{0pt}stable theories \cite{Shelah1972}. In an unpublished result, Ressayre was able to show uniqueness of a special kind of prime model without any assumptions regarding the theory, and the technique of this proof was extended greatly in stability theory. (See \cite[Sec.~X.3]{Baldwin1987} for a discussion of the history of these results.)

\begin{defn}\label{defn:construction-sequence}
  A sequence $(b_i)_{i < \alpha}$ of elements is a \emph{construction sequence over $A$} if for each $i$, $\tp(b_i/Ab_{<i})$ is atomic. A model $M\supseteq A$ is \emph{constructible over $A$} if it can be enumerated by a construction sequence over $A$. $M$ is \emph{constructible} if it is constructible over $\varnothing$.
\end{defn}

\begin{thm}[Ressayre]\label{thm:Ressayre}
  Any two constructible models of a complete theory are isomorphic.
\end{thm}

In the context of continuous logic, Vaught's result is known to generalize \cite[Cor.~12.9]{MTFMS}, but the proof requires some modification and prime models are only guaranteed to be `approximately $\aleph_0$-homogeneous,' rather than $\aleph_0$-homogeneous. The modified proof uses the kind of back-and-forth-with-error argument that is common in continuous logic but also seems to rely on working with constructions of length $\omega$. This can sometimes cause difficulty when trying to generalize such arguments to constructions of uncountable length, such as the back-and-forth argument in the proof of \cref{thm:Ressayre}.

Another difficulty of generalizing certain results to continuous logic is the following phenomenon: In discrete logic, one can show that for any $a$ and $b$, $\tp(ab)$ is atomic if and only if $\tp(b)$ and $\tp(a/b)$ are atomic, and this fact is typically used in proofs of  \cref{thm:Ressayre}. In continuous logic, we still have that if $\tp(b)$ and $\tp(a/b)$ are atomic, then $\tp(ab)$ is atomic, but the converse does not in general need to hold. For a simple example, consider a two-sorted structure $(A,B,f)$ such that $A$ is the circle $S^1$ with the $\{0,1\}$-valued discrete metric, $B$ is the circle with its ordinary path metric, and $f: A \to B$ is the identity map. For any $a \in A$ and $b = f(a)$, we have $\tp(ab)$ and $\tp(a)$ are atomic, but $\tp(a/b)$ is not. This phenomenon is the root of many failures of continuous generalizations classical results, such as Vaught's never-two theorem and Lachlan's theorem on the number of countable models of a superstable theory, which was investigated extensively in \cite{Yaacov2007} using the notion of $d$-finiteness of types.

These difficulties might lead one to suspect that \cref{thm:Ressayre} does not generalize to continuous logic, but, as evidenced by the wide array of extensions of Ressayre's technique in stability theory, the proof is robust enough to be salvageable.

We would like to thank Ita\"i Ben Yaacov for some valuable discussion regarding this problem.

\section{Augmented construction sequences}

Note that while we originally stated \cref{defn:construction-sequence} for discrete logic, the definition makes sense verbatim in continuous logic. Recall that a type $\tp(b/A)$ (with $b$ a finite tuple\footnote{Note that the elements of a construction sequence are required to be singletons.}) is \emph{atomic} if there is an $A$-definable predicate $D(x)$ such that for any $c$ in the monster, $D(c) = \inf\{d(c,b') : b'\equiv_A b\}$. As discussed in the introduction, it is possible to show that if $\tp(bc/A)$ is atomic, then $\tp(b/A)$ is atomic. %

\begin{lem}\label{lem:construction-embedding}
  For any set of parameters $A$, construction sequence $(b_i)_{i <\alpha}$ over $A$, and model $M \supseteq A$, there is an elementary map $f: Ab_{<\alpha} \to M$ that fixes $A$ pointwise. %
\end{lem}
\begin{proof}
  By a routine argument, atomic types must always be realized in models. The lemma then follows by transfinite induction.
\end{proof}

\begin{defn}
  Given a continuous first-order theory $T$ in a language $\Lc$ and a set of parameters $A$, an \emph{augmented construction sequence over $A$} is a sequence $(b_i,\varphi_i,C_i,\Sc_i)_{i < \alpha}$ such that $(b_i)_{i<\alpha}$ is a construction sequence over $A$ and for each $i<\alpha$, $\varphi_i$ is the distance predicate of $\tp(b_i/Ab_{<i})$, $C_i \subseteq Ab_{<i}$ is a countable set of parameters $\Sc_i\subseteq \Lc$ is a countable language such that $\varphi_i$ is definable in $\Sc_i$ over $C_i$.

\end{defn}

We will also sometimes refer to augmented construction sequences indexed by a subset of an ordinal (i.e., $(b_i,\varphi_i,C_i,\Sc_i)_{i \in X}$ for $X \subseteq \alpha$) with essentially the same definition. Most of our lemma will only be stated and proven in ordinal-indexed case for the sake of notational simplicity, but there is no subtlety in generalizing these statements to sequences indexed by subsets of ordinals, since any subset of an ordinal is canonically order-isomorphic to an ordinal.

Note that it is immediate that any construction sequence can be extended to an augmented construction sequence. Note, moreover, that $\varphi_i$ is uniquely determined (up to logical equivalence) by $\tp(b_i/Ab_{<i})$, but $C_i$ and $\Sc_i$ are not. We have put the $\varphi_i$'s in explicitly as a bookkeeping device that will be useful later when we show that under certain circumstances, $(b_i,\varphi_i,C_i,\Sc_i)_{i \in X}$ is still an augmented construction sequence (possibly over a different set of parameters) for some $X \subseteq \alpha$. Here it could in principle be the case that the same $C_i$'s and $\Sc_i$'s witness the construction but with different distance predicates, but it will be important that the distance predicates do not actually change. That said, we will also frequently write $(b_i,\varphi_i,C_i,\Sc_i)_{i <\alpha}$ as $(b_i,C_i,\Sc_i)_{i < \alpha}$ when we don't need to emphasize that we are keeping track of specific distance predicates.

\begin{defn}
  For any augmented construction sequence $\Sbf = (b_i,C_i,\Sc_i)_{i < \alpha}$ over $A$, we say that a pair $(A',X)$ with $A' \subseteq A$ and $X \subseteq \alpha$ is \emph{construction-closed in $\Sbf$} if for every $i \in X$, $C_i \subseteq A' \cup \{b_i : i \in X\}$. We say that $(A',X)$ is \emph{countable} if $|A'|+|X| \leq \aleph_0$.
\end{defn}

\begin{lem}\label{lem:const-closed-restrict}
  For any augmented construction sequence $\Sc = (b_i,C_i,\Sc_i)_{i < \alpha}$ and any construction-closed pair $(A',X)$ in $\Sc$, $\Sc\res X$ is an augmented construction sequence over $A'$.
\end{lem}
\begin{proof}
  This is immediate from the relevant definitions.
\end{proof}

\begin{defn}
  $b$ and $c$ are \emph{weakly orthogonal over $A$} if for any $b' \equiv_A b$ and $c' \equiv_A c$, $b'c'\equiv_A bc$. We write $b \indwo_A c$ to denote that $b$ and $c$ are weakly orthogonal over $A$.
\end{defn}

Note that $\indwo$ is clearly symmetric by definition.

The following lemma is what we will use in lieu of the standard fact in discrete logic that $\tp(ab)$ is atomic if and only if $\tp(b)$ is atomic and $\tp(a/b)$ is atomic. For any set $A$ in a metric space, we will write $\ol{A}$ for the metric closure of $A$.

\begin{lem}\label{lem:can't-believe-it's-not}
  Fix a set $A$.
  \begin{enumerate}
  \item\label{wo-closedness} For any $B$ and $C$, $B \indwo_A C$ if and only if $\ol{B} \indwo_{\ol{A}}\ol{C}$.
  \item\label{entails-0} For any $B$ and $C$, the following are equivalent.
    \begin{enumerate}
    \item\label{wo} $B \indwo_A C$
    \item\label{forward} $\tp(C/A) \vdash \tp(C/AB)$
    \item\label{backward} $\tp(B/A)\vdash \tp(B/AC)$
    \end{enumerate}
  \item\label{wo-trans} For any $B$, $C$, and $D$, if $D \indwo_A B$ and $D \indwo_{AB}C$, then $D \indwo_A BC$.
  \item\label{atomic-implies-wo} For any finite tuple $b$ and $C$, if $\tp(b/AC)$ is atomic with an $A$-definable distance predicate, then $b \indwo_A C$.
  \item\label{entails-1} For any $B$ and sequence $(c_i)_{i<\alpha}$, if $c_i \indwo_{Ac_{<i}}B$ for each $i<\alpha$, then $c_{<\alpha}\indwo_A B$.
  \item\label{is-atomic} For any finite tuple $b$ and $C$, if $b \indwo_A C$ and $\tp(b/A)$ is atomic, then $\tp(b/AC)$ is atomic with an $A$-definable distance predicate.
  \end{enumerate}
\end{lem}
\begin{proof}
  \ref{wo-closedness} follows immediately from the definition of $\indwo$ and the fact that $B \equiv_A B'$ if and only if $\ol{B} \equiv_{\ol{A}} \ol{B'}$ (where we choose compatible enumerations of $\ol{B}$ and $\ol{B'}$).

  For \ref{entails-0}, by symmetry we just need to show that \ref{wo} and \ref{forward} are equivalent. \ref{forward} is equivalent to the following: For any $C' \equiv_A C$, $C' \equiv_{AB}C$. Since $C' \equiv_{AB}C$ is equivalent to $BC' \equiv_A BC$, we clearly have that \ref{wo} implies \ref{forward}. Now assume \ref{forward} and fix $B' \equiv_A B$ and $C' \equiv_A C$. Fix an automorphism $\sigma$ (fixing $A$ pointwise) such that $\sigma(B') = B$. We now have that $\sigma(C') \equiv_A C$, whereby $(B,\sigma(C')) \equiv_A BC$. Applying $\sigma^{-1}$ then gives $B'C' \equiv_A BC$.

  For \ref{wo-trans}, note that if $\tp(D/A) \vdash \tp(D/AB)$ and $\tp(D/AB)\vdash \tp(D/ABC)$, then $\tp(D/A)\vdash \tp(D/ABC)$. The result now follows from \ref{entails-0}.
  
  For \ref{atomic-implies-wo}, the given condition clearly implies that $\tp(b/A)\vdash \tp(b/AC)$, so by \ref{entails-0}, we have that $b \indwo_A C$.
  
  For \ref{entails-1}, we will prove by induction that for each $i < \alpha$, $\tp(c_{<i}/A)\vdash \tp(c_{<i}/AB)$. Assume that we have shown this for all $j < i$. If $i$ is a limit ordinal or $0$, we immediately have that $\tp(c_{<i}/A)\vdash \tp(c_{<i}/AB)$. So let $i = k+1$. We have by the induction hypothesis that $\tp(c_{<k}/A)\vdash \tp(c_{<k}/AB)$. Therefore by \ref{entails-0}, we have that $\tp(B/A)\vdash \tp(B/Ac_{<k})$. By assumption, $\tp(c_k/Ac_{<k}) \vdash \tp(c_k/ABc_{<k})$, so by \ref{entails-0} again, we have that $\tp(B/Ac_{<k})\vdash \tp(B/Ac_{< k+1})$, whereby $\tp(B/A) \vdash \tp(B/Ac_{<k+1})$. By \ref{entails-0} a third time, $\tp(c_{<k+1}/A)\vdash \tp(c_{<k+1}/AB)$.

  For \ref{is-atomic}, we have that the set of realizations (in the monster) of $\tp(b/A)$ is $A$-definable and we have that the set of realizations of $\tp(b/A)$ is the same as the set of realizations (in the monster) of $\tp(b/AC)$. Therefore $\tp(b/AC)$ is atomic and its distance predicate is $A$-definable. %
\end{proof}

For any $X \subseteq \alpha$, we will write $\Sbf\res X$ for the restricted sequence $(b_i,\varphi_i,C_i,\Sc_i)_{i \in X}$. We will also abbreviate $A' \cup \{b_i : i \in X\}$ as $A'b_{\in X}$ and $\bigcup_{i \in X}\Sc_i$ as $\Sc_X$.

\begin{prop}\label{prop:construction-shuffle}
  \begin{sloppypar}
    Fix an augmented construction sequence $\Sbf = (b_i,\varphi_i,C_i,\Sc_i)_{i < \alpha}$ over $A$. For any construction-closed pair $(A',X)$ in $\Sbf$, $\Sbf \res (\alpha \setminus X)$ is also an augmented construction sequence over $Ab_{\in X} = A\cup \{b_i : b \in X\}$.
  \end{sloppypar}
\end{prop}
\begin{proof}
  For any $i < \alpha$, let $B_i = \{b_j : j \in X \vee j < i\}$. We need to argue that for any $i < \alpha$ with $i \notin X$, $\tp(b_i/AB_i)$ is atomic and moreover its distance predicate is $\Sc_i$-definable over $C_i$. We will argue that the distance predicate of $\tp(b_i/AB_i)$ is the same as the distance predicate of $\tp(b_i/Ab_{<i})$.

  \begin{sloppypar}
    Fix $i < \alpha$ with $i \notin X$. Let $Y = \{j \in X: j < i\}$ and $Z = \{j \in X : j > i\}$. (Note that $X = Y \cup Z$.) Let $\beta$ be the order type of $Z$ and let $(e_j)_{j<\beta}$ be an enumeration of $(b_i)_{i \in Z}$ in order. We have by construction that $\tp(e_j/Ab_{< i}b_ie_{<j})$ is atomic with an $Ab_{<i}e_{<j}$-definable distance predicate. Therefore $e_j \indwo_{Ab_{<i}e_{<j}}b_i$ and by \cref{lem:can't-believe-it's-not} part \ref{entails-1}, we have that $e_{<\beta} \indwo_{Ab_{<i}}b_i$. By symmetry, $b_i \indwo_{Ab_{<i}}e_{<\beta}$. Since $\tp(b_i/Ab_{<i})$ is also atomic, we now have by \cref{lem:can't-believe-it's-not} part \ref{is-atomic} that $\tp(b_i/Ab_{<i}e_{<\beta})$ is atomic with the same distance predicate. Note that $Ab_{<i}e_{<\beta}$ is the same set as $A\cup b_{\in X}\cup \{b_j: j < i \wedge j \notin X\}$. Therefore we have that $\tp(b_i/A\cup b_{\in X}\cup \{b_j: j < i \wedge j \notin X\})$ is atomic with the same distance predicate as $\tp(b_i/Ab_{<i})$, namely $\varphi_i$.
  \end{sloppypar}
Since we can do this for any $i < \alpha$ with $i \notin X$, we have that $\Sbf\res (\alpha \setminus X)$ is an augmented construction sequence over $Ab_{\in X}$.
\end{proof}

We will also eventually need the following lemmas.

\begin{lem}\label{lem:const-seq-weak-orth}
  Fix an augmented construction sequence $\Sc = (b_i,\varphi_i,C_i,\Sc_i)_{i < \alpha}$ over $A$. For any construction-closed pair $(B,X)$, $b_{\in X}\indwo_B A$. 
\end{lem}
\begin{proof}
  Let $(e_j)_{j < \beta}$ be an enumeration of $b_{\in X}$ in order. We will prove the statement by induction on $j < \beta$. Assume that we have shown that $e_{<j} \indwo_B A$. By construction, $\tp(e_j/Ae_{<j})$ is atomic with a $Be_{<j}$-definable distance predicate. Therefore $e_j \indwo_{Be_{<j}}A$. Fix a tuple $f_{<j+1}$ and assume that $f_{<j+1}\equiv_B e_{<j+1}$. By the induction hypothesis, we have that $f_{<j} \equiv_A e_{<j}$. Fix an automorphism $\sigma$ of the monster fixing $A$ pointwise and taking $f_{<j}$ to $e_{<j}$. Note that $\sigma$ also fixes $B$ pointwise. We now have that $\sigma(f_{<j+1}) \equiv_{B} e_{<j+1}$. Since $\sigma(f_{<j}) = e_{<j}$, we have $\sigma(f_j) \equiv_{B e_{<j}}e_j$. Therefore $\sigma(f_j) \equiv_{Ae_{<j}}e_j$, and so $f_{<j+1}\equiv_A\sigma(f_{<j+1})\equiv_A e_{<j+1}$. Since we can do this for any such $f_{<j+1}$, we have that $\tp(e_{<j+1}/B)\vdash \tp(e_{<j+1}/A)$ and therefore $e_{<j+1}\indwo_B A$.

  Limit stages are immediate, so the statement in the lemma follows.
\end{proof}

\begin{lem}\label{lem:asc-chain-ok}
  Fix an augmented construction sequence $\Sc = (b_i,\varphi_i,C_i,\Sc_i)_{i < \alpha}$ over $A$. Let $(X_j)_{j < \beta}$ be an ascending sequence of subsets of $\alpha$ such that for each $j$, $(A,X_j)$ is construction-closed in $\Sc$ and $\Sc\res(\alpha \setminus X_j)$ is an augmented construction sequence over $Ab_{ \in X_j} = A \cup \{b_i : i \in X_j\}$. Then $\Sc \res (\alpha \setminus \bigcup_{j < \beta} X_j)$ is an augmented construction sequence over $A \cup \{b_i : i \in \bigcup_{j < \beta}X_j\}$.
\end{lem}
\begin{proof}
  If $\beta$ is not a limit ordinal, then this is trivial, so assume that $\beta$ is a limit ordinal.

  Note that for any $i < \alpha$, we have $C_i \subseteq A \cup \{b_k : k \in X_j\} \cup \{b_k : k < i,~k \in X_j \}$ (since the set on the right-hand side always contains $A b_{<i}$). Therefore all we really need to check is that $\varphi_i$ is the distance predicate of $p \coloneq \tp(b_i / A \cup \{b_k : k \in \bigcup_{j < \beta}X_j\} \cup \{b_k : k < i,~k \in \bigcup_{j<\beta}X_j\})$. Since $\varphi_i$ is the distance predicate of $q_k\coloneq \tp(b_i/ A \cup \{b_k : k \in X_j\} \cup \{b_k : k < i,~k \in X_j \})$ for each $k < \beta$, we have that the set of realizations of $q_k$ in the monster model does not depend on $k$. Therefore it is also the same as the set of realizations of $p$ in the monster model, whereby $\varphi_i$ is still the distance predicate of $p$.
\end{proof}

\section{Self-sufficiency}

\begin{lem}\label{lem:res-still-aug}
  Fix an augmented construction sequence $\Sbf = (b_i,\varphi_i,C_i,\Sc_i)$ over $A$. For any construction-closed pair $(A',X)$ in $\Sbf$, $\Sbf\res X$ is an augmented construction sequence over $A'$ (relative to the theory $T \res \Sc_X$).
\end{lem}
\begin{proof}
  Each $\varphi_i$ for $i \in X$ is still a distance predicate in the $\Sc_X$-reduct and so still isolates $\tp(b_i/A\cup\{b_j: j < i,~j \in X\})$ for each $i \in X$.
\end{proof}

\begin{defn}
  Fix an augmented construction sequence $\Sbf = (b_i,\varphi_i,C_i,\Sc_i)_{i < \alpha}$ with $(b_i)_{i<\alpha}$ an enumeration of some constructible model $M$ over $A$. A pair $(A',X)$ with $A' \subseteq A$ and $X \subseteq \alpha$ is \emph{self-sufficient in $\Sbf$} if it is construction-closed in $\Sbf$ and  $A'b_{\in X}$ is a dense subset of an $\Sc_X$-elementary substructure of $M$.
\end{defn}

\begin{lem}\label{lem:self-suff-is-atomic-mod}
  Fix an augmented construction sequence $\Sbf$. For any countable self-sufficient pair $(A',X)$ in $\Sbf$, $A'b_{\in X}$ is a dense subset of the unique separable atomic model of $(T \res \Sc_X)_{\ol{A'}}$ (i.e., the $\Sc_X$-reduct of $T$ with constants added for $\ol{A'}$).\footnote{Note that although $(T \res \Sc_X)_{\ol{A'}}$ is not necessarily a theory in a countable language, it is interdefinable with a theory in a countable language (namely $(T \res \Sc_X)_{A'}$), so general facts regarding uniqueness of separable atomic models still apply to it.}
\end{lem}
\begin{proof}
  By \cref{lem:res-still-aug}, we have that $\Sbf\res X$ is an augmented construction sequence relative to $T \res \Sc_X$. Therefore, by \cref{lem:construction-embedding}, we have that for any model $M \models (T \res \Sc_X)_{A'}$, there is an elementary map $f : A'b_{\in X} \to M$ (that fixes $A'$ pointwise). This extends to an elementary embedding of $\ol{A'b_{\in X}}$. Since we can do this for any model $M$, we have that $\ol{A'b_{\in X}}$ is a separable prime model of $(T \res \Sc_X)_{A'}$. Therefore by \cite[Cor.~12.9]{MTFMS} it is the unique separable atomic model of $(T \res \Sc_X)_{A'}$. Since every element of $\ol{A'}$ is definable over $A'$, we have that it is the unique separable atomic model of $(T\res \Sc_X)_{\ol{A'}}$ as well.
\end{proof}

It is routine to show that for any augmented construction sequence $\Sbf$ over $A$ (with length $\alpha$) and any pair $(D,Y)$ with $D \subseteq A$ and $Y \subseteq \alpha$, there is a unique smallest construction-closed pair $(B,X)$ such that $B \supseteq D$ and $X \supseteq Y$. Moreover, if $(D,Y)$ is countable, then this $(B,X)$ will be as well. (This follows from the fact that a countably branching well-founded tree is countable.) We will call a pair $(B,X)$ like this the \emph{construction-closure} of $(D,Y)$.

\begin{defn}
  Given sets of parameters $B^0$ and $B^1$, a \emph{densely defined isomorphism between $B^0$ and $B^1$} is an elementary map $f: B^0 \to \ol{B^1}$ with dense image.
\end{defn}

Given a densely defined isomorphism $f$ between $B^0$ and $B^1$, we clearly have that $f$ has a unique continuous extension to $\ol{B^0}$ and that this is an elementary bijection between $\ol{B^0}$ and $\ol{B^1}$. We will also write this as $f$ and write its inverse as $f^{-1}$.

We will now prove our main technical lemma.

\begin{lem}\label{lem:self-suff-exists-duo}
  Fix constructible models $M^0$ and $M^1$ of the same complete theory. Fix augmented construction sequences $\Sbf^0=(b_i^0,C_i^0,\Sc_i^0)_{i < \alpha^1}$ over $A^1$ and $\Sbf^1=(b_i^1,C_i^1,\Sc_i^1)_{i < \alpha^1}$ over $A^1$ enumerating $M^0$ and $M^1$, respectively. Fix a densely defined isomorphism $f : A^0 \to \ol{A^1}$.

  For any countable $D^0 \subseteq A^0$, $Y^0 \subseteq \alpha^0$, $D^1 \subseteq A^1$, and $Y^1 \subseteq \alpha^1$, there are countable self-sufficient pairs $(B^0,X^0)$ and $(B^1,X^1)$ with $B^0 \supseteq D^0$, $X^0 \supseteq Y^0$, $B^1 \supseteq D^1$, and $X^1 \supseteq Y^1$ such that $\Sc_{X^0}^0 = \Sc_{X_1}^1$ (i.e., $\bigcup_{i \in X^0}\Sc^0_i = \bigcup_{i \in X^1}\Sc^1_i$) and $f\res B^0$ is a densely defined isomorphism between $B^0$ and $B^1$. 
\end{lem}
\begin{proof}
  Let $D_0^0 = D^0$, $Y_0^0 = Y^0$, $D_0^1 = D^1$, and $Y_0^1 = Y^1$. Given countable $(D_n^0,Y_n^0)$ and $(D_n^1,Y_n^1)$, build $(D_{n+1}^0Y_{n+1}^0)$ and $(D_{n+1}^1,Y_{n+1}^1)$ as follows:
  \begin{itemize}
  \item Find a separable $\Sc_{Y^0_{n}}^0\cup \Sc_{Y^1_{n}}^1$-elementary substructure $N_n^0 \subseteq M^0$ containing $D_n^0\cup b^0_{\in Y^0_n}$ and a countable dense subset of $A^0 \cap \ol{f^{-1}(D_n^1)}$.
  \item  Likewise, find a separable $\Sc_{Y^0_{n}}^0\cup \Sc_{Y_{n}}^1$-elementary substructure $N_n^1 \subseteq M^1$ containing $D_n^1 \cup b^1_{\in Y^1_n}$ and a countable dense subset of $A^1 \cap \ol{f(D_n^0)}$.
  \item Find countable $(D^0_{n + \nicefrac{1}{2}},Y^0_{n + \nicefrac{1}{2}})$ such that $D^0_{n + \nicefrac{1}{2}}b^0_{\in Y^0_{n+\nicefrac{1}{2}}}$ is dense in $N^0_n$. Similarly, find countable $(D^1_{n + \nicefrac{1}{2}},Y^1_{n + \nicefrac{1}{2}})$ such that $D^1_{n + \nicefrac{1}{2}}b^1_{\in Y^1_{n+\nicefrac{1}{2}}}$ is dense in $N^1_n$. %
    \begin{sloppypar}
    \item Let $(D^0_{n+1},Y^0_{n+1})$ be the construction-closure of $(D^0_{n + \nicefrac{1}{2}},Y^0_{n + \nicefrac{1}{2}})$, and let $(D^1_{n+1},Y^1_{n+1})$ be the construction-closure of $(D^1_{n + \nicefrac{1}{2}},Y^1_{n + \nicefrac{1}{2}})$.
    \end{sloppypar}
  \end{itemize}
  Note that $(D^0_{n+1},Y^0_{n+1})$ and $(D^1_{n+1},Y^1_{n+1})$ are still countable. Finally let $B^0 = \bigcup_{n < \omega}D^0_n$, $X^0 = \bigcup_{n < \omega}Y^0_n$, $B^1 = \bigcup_{n<\omega}D^0_n$, and $X^1 = \bigcup_{n<\omega}Y^1_n$. It is immediate from the definition that $(B^0,X^0)$ is construction-closed in $\Sbf^0$ and $(B^1,X^1)$ in $\Sbf^1$. Furthermore, we have that $Bb^0_{\in X} = \bigcup_{n < \omega}D^0_{n+\nicefrac{1}{2}}b^0_{\in Y^0_{n+\nicefrac{1}{2}}}$ is (for every $n<\omega$) a dense subset of an $\Sc^0_{Y_n}\cup\Sc^1_{Y_n}$-elementary substructure of $M^0$ (specifically $\ol{\bigcup_{n < \omega}N^0_n}$). Therefore $Bb^0_{\in X^0}$ is a dense subset of an $\Sc_{X^0}$-elementary substructure of $M^0jk$. Hence $(B^0,X^0)$ is self-sufficient in $\Sbf^0$. By the same argument $(B^1,X^1)$ is self-sufficient in $\Sbf^1$. Finally, by construction, we have that $\Sc_{X^0}^0 = \bigcup_{n < \omega}\Sc^0_{Y^0_n}\cup \Sc^1_{Y^1_n} = \Sc_{X^1}^1$ and that $f(B^0)$ is dense in $\ol{B^1}$, whereby $(f\res B^0) : B^0 \to \ol{B^1}$ is a densely defined isomorphism.
\end{proof}

\section{Uniqueness of constructible models}

\begin{lem}\label{lem:S-iso-is-iso}
  Fix augmented construction sequences $\Sbf^0 = (b_i^0,C_i^0,\Sc_i^0)_{i < \alpha^0}$ over $A^0$ and $\Sbf^1=(b_i^1,C_i^1,\Sc_i^1)_{i < \alpha^1}$ over $A^1$ with $A^0$ and $A^1$ countable and such that $\Sbf^0$ and $\Sbf^1$ enumerate some models $M^0$ and $M^1$, respectively. Fix self-sufficient pairs $(A^0,X^0)$ in $\Sbf^0$ and $(A^1,X^1)$ in $\Sbf^1$. If $\Sc\coloneq \Sc^0_{X^0} = \Sc^1_{X^1}$, then any dense $\Sc$-isomorphism $g : A^0b^0_{\in X^0} \to \ol{A^1b^1_{\in X^1}}$ is a dense $\Lc$-isomorphism.
\end{lem}
\begin{proof}
  By self-sufficiency, we know that the $\Lc$-types of every finite tuple of elements of $A^0b^0_{\in X^0}$ and $A^1b^1_{\in X^1}$ are determined by their $\Sc$-types. Therefore $g$ is an $\Lc$-elementary map and so is a dense $\Lc$-isomorphism.
\end{proof}

Given our notation of $\Sc_X$ for the language $\bigcup_{i \in X}\Sc_i$, we will write $\Sc(B)$ for the language $\Sc$ with constants added for the elements of a set of parameters $B$.

\begin{lem}\label{lem:dense-iso-extend}
  Fix augmented construction sequences $\Sbf^0 = (b_i^0,C_i^0,\Sc_i^0)_{i < \alpha^0}$ over $A^0$ and $\Sbf^1=(b_i^1,C_i^1,\Sc_i^1)_{i < \alpha^1}$ over $A^1$ with $A^0$ and $A^1$ countable and such that $\Sbf^0$ and $\Sbf^1$ enumerate some models $M^0$ and $M^1$, respectively. Fix countable sets $X^0\subseteq \alpha^0$ and $X^1 \subseteq \alpha^1$ such that $(A^0,X^0)$ and $(A^1,X^1)$ are self-sufficient in $\Sbf^0$ and $\Sbf^1$. If $\Sc^0_{X^0} = \Sc^1_{X^1}$, then for any densely defined isomorphism $f:A^0 \to \ol{A^1}$, there is a densely defined isomorphism $g: A^0b^0_{\in X^0} \to \ol{A^1b^1_{\in X^1}}$ that extends $f$.
\end{lem}
\begin{proof}
  Let $\Sc = \Sc^0_{X^0} = \Sc^1_{X^1}$. By applying an automorphism of the monster, we may assume that $\ol{A^0} = \ol{A^1}$ and that $f: A^0 \to \ol{A^1}$ is the identity map. By \cref{lem:self-suff-is-atomic-mod}, the $\Sc(\ol{A^0})$-reduct of $\ol{A^0b^0_{\in X^0}}$ is the unique separable atomic model of $(T\res \Sc)_{\ol{A^0}}$. Likewise, the $\Sc(\ol{A^1})$-reduct of $\ol{A^1b^1_{\in X^1}}$ is the unique separable atomic model of $(T\res \Sc)_{\ol{A^1}}$. Since $\ol{A^0} = \ol{A^1}$, $\Sc(\ol{A^0})$ and $\Sc(\ol{A^1})$ are the same language and $(T\res \Sc)_{\ol{A^0}}$ and $(T\res \Sc)_{\ol{A^1}}$ are the same theory. Therefore, we have that there is an $\Sc$-isomorphism between $\ol{A^0b^0_{\in X^0}}$ and $\ol{A^1b^1_{\in X^1}}$ that fixes $\ol{A^0} = \ol{A^1}$ pointwise. By \cref{lem:S-iso-is-iso}, this is a dense $\Lc$-isomorphism.
\end{proof}

\begin{lem}\label{lem:dense-iso-amalg}
  Fix sets $B^0,C^0 \supseteq A^0$ and $B^1,C^1 \supseteq A^1$ and densely defined isomorphisms $f : B^0 \to \ol{B^1}$ and $g: C^0 \to \ol{C^1}$ such that $f \res A^0 = g\res A^0$. If $B^0 \indwo_{A^0}C^0$, then $f \cup g$ is a densely defined isomorphism between $B^0\cup C^0$ and $B^1\cup C^1$. 
\end{lem}
\begin{proof}
  We clearly still have that $f(B^0)\indwo_{f(A^0)}g(C^0)$. By \cref{lem:can't-believe-it's-not} part \ref{wo-closedness}, we have that $\ol{f(B^0)}\indwo_{\ol{f(A^0)}}\ol{g(C^0)}$. Since $\ol{f(B^0)} = \ol{B^1}$, $\ol{f(A^0)} = \ol{A^1}$, and $\ol{g(C^0)} = \ol{C^1}$, we have $B^1 \indwo_{A^1}C^1$ by \cref{lem:can't-believe-it's-not} part \ref{wo-closedness} again.
\end{proof}

\begin{thm}\label{thm:main-theorem}
  Any two constructible models of a complete continuous first-order theory are isomorphic.
\end{thm}
\begin{proof}
  Fix constructible models $M^0$ and $M^1$ of a complete theory $T$. Fix augmented construction sequences $\Sbf^0 = (b_i^0,\varphi_i^0,C_i^0,\Sc_i^0)_{i < \alpha^0}$ and $\Sbf^1 = (b_i^1,\varphi_i^1,C_i^1,\Sc_i^1)_{i < \alpha^1}$ over $\varnothing$ enumerating $M^0$ and $M^1$ respectively.

  We will build by transfinite induction sets $Z^0_j \subseteq \alpha^0$ and $Z^1_j \subseteq \alpha^1$ and densely defined isomorphisms $f_j : b_{\in Z^0_j}^0 \to \ol{b_{\in Z^1_j}^1}$ satisfying that
  \begin{itemize}
  \item for $j<k$, $Z^0_j \subseteq Z^0_k$, $Z^1_j \subseteq Z^1_k$, and $f_j$ is extended by $f_k$,
  \item if $j < \alpha^0$, then $b_j^0 \in Z^0_{j+1}$ and if $j < \alpha^1$, then $b_j^1 \in Z^1_{j+1}$, and
  \item[$\star$] for each $j$, $\Sbf^0\res (\alpha^0 \setminus Z^0_j)$ is an augmented construction sequence over $b_{\in Z^0_j}^0$ and $\Sbf^1 \res (\alpha^1\setminus Z^1_j)$ is an augmented construction sequence over $b_{\in Z^1_j}^1$.
  \end{itemize}
  Assume that we have built this up to some limit ordinal $j$. If we let $Z^0_j = \bigcup_{k < j}Z^0_k$ and likewise for $Z^1_j$ and $f_j$, then it is easy to check that $f_j : b^0_{\in Z^0_j} \to \ol{b^1_{\in Z^1_j}}$ is a densely defined isomorphism. Furthermore, $\star$ holds by \cref{lem:asc-chain-ok}.
  This means that we only really need to worry about successor stages.

  Assume that we have $Z^0_j$, $Z^1_j$, and a densely defined isomorphism $f_j : b^0_{\in Z^0_j} \to \ol{b^1_{\in Z^1_j}}$ such that $\star$ holds. Apply \cref{lem:self-suff-exists-duo} %
  to the construction sequences $\Sbf^0\res (\alpha^0\setminus Z^0_j)$ and $\Sbf^1\res(\alpha^1 \setminus Z^1_j)$ to get countable self-sufficient pairs $(A^0_j,X^0_j)$ and $(A^1_j,X^1_j)$ such that
  \begin{itemize}
  \item $A^0_j \subseteq b^0_{\in Z^0_j}$ and $A^1_j \subseteq b^1_{\in Z^1_j}$,
  \item $X^0_j\subseteq \alpha^0\setminus Z^0_j$ and $X^1_j \subseteq \alpha^1\setminus Z^1_j$,
  \item if $j \in \alpha^0 \setminus Z^0_j$, then $j \in X^0_j$ and if $j \in \alpha^1 \setminus Z^1_j$, then $j \in X^1_j$,
  \item $\Sc^0_{X^0_j} = \Sc^1_{X^1_j}$, and
  \item $f_j \res A^0_j$ is a densely defined isomorphism between $A^0_j$ and $A^1_j$.
  \end{itemize}
  By \cref{lem:dense-iso-extend}, we can find a densely defined isomorphism $g_j : A^0_jb^0_{\in Z^0_j} \to \ol{A^1_j b^1_{\in Z^1_j}}$ extending $f_j\res A^0_j$. We now need to argue that $f_j \cup g_j$ is a densely defined isomorphism between $\{b^0_k : k \in Z^0_j \cup X^0_j\}$ and $\{b^1_k : k \in Z^1_j \cup X^1_j\}$. Since $A^0_j \subseteq b^0_{\in Z^0_j}$, it is immediate that $(b^0_{\in Z^0_j},X^0_j)$ is a construction-closed pair. By \cref{lem:const-closed-restrict}, we have that $(\Sc\res \alpha^0\setminus Z^0_j) \res X^0_j = \Sc \res X^0_j$ is an augmented construction sequence over $b^0_{\in Z^0_j}$. Moreover, $(A^0_j,X^0_j)$ is construction-closed in $\Sc\res X^0_j$ over $b^0_{\in Z^0_j}$. By Lemmas~\ref{lem:const-seq-weak-orth} and \ref{lem:can't-believe-it's-not}, $b^0_{\in X^0_j} \indwo_{A^0_j}b^0_{\in Z^0_j}$. Therefore by \cref{lem:dense-iso-amalg}, we have that $f_j \cup g_j$ is a densely defined isomorphism between $b^0_{\in X^0_j}\cup b^0_{\in Z^0_j}$ and $b^1_{\in X^1_j}\cup b^1_{\in Z^1_j}$. Finally note that $\Sbf^0\res (\alpha^0\setminus Z^0_{j+1})$ is an augmented construction sequence over $b^0_{\in Z^0_{j+1}}$ by \cref{prop:construction-shuffle}, and likewise for $\Sbf^1\res (\alpha^1\setminus Z^1_{j+1})$ over $b^1_{\in Z^1_{j+1}}$.

Since can run this construction indefinitely, we will have that $f_{\max\{\alpha^0,\alpha^1\}}$ is an isomorphism between $M^0$ and $M^1$. 
\end{proof}

Note that an immediate corollary of \cref{thm:main-theorem} is that if $M^0$ and $M^1$ are both constructible models of the same complete theory over a set of parameters $A$, then there is an isomorphism between $M^0$ and $M^1$ fixing $A$ pointwise.

\bibliographystyle{plain}
\bibliography{../ref}

\end{document}